\documentclass[11pt]{amsart}

\usepackage[margin=0.85in]{geometry}
\usepackage[T1]{fontenc}
\usepackage{lmodern}
\usepackage{amsmath,amssymb,mathtools}
\usepackage{microtype}
\usepackage{enumitem}
\usepackage[unicode,hidelinks,bookmarksopen=true]{hyperref}

\newcommand{\rev}[1]{#1}
\newenvironment{revision}{}{}
\newcommand{\dens}{\operatorname{dens}}
\newcommand{\rank}{\operatorname{rank}}
\newcommand{\dist}{\operatorname{dist}}
\newcommand{\dpr}{d_{\mathrm{pr}}}
\newcommand{\cG}{\mathcal G}
\newcommand{\cE}{\mathcal E}
\newcommand{\cS}{\mathcal S}
\newcommand{\R}{\mathbb R}
\newcommand{\Pp}{\mathbb P}
\newcommand{\Ee}{\mathbb E}
\newcommand{\1}{\mathbf 1}

\newcommand{\thmref}[1]{\hyperref[#1]{Theorem~\ref*{#1}}}
\newcommand{\propref}[1]{\hyperref[#1]{Proposition~\ref*{#1}}}
\newcommand{\lemref}[1]{\hyperref[#1]{Lemma~\ref*{#1}}}
\newcommand{\corref}[1]{\hyperref[#1]{Corollary~\ref*{#1}}}
\newcommand{\remref}[1]{\hyperref[#1]{Remark~\ref*{#1}}}
\newcommand{\secref}[1]{\hyperref[#1]{Section~\ref*{#1}}}
\newcommand{\secrefs}[2]{%
  \hyperref[#1]{Sections~\ref*{#1}} and \hyperref[#2]{\ref*{#2}}}

\newtheorem{theorem}{Theorem}[section]
\newtheorem{proposition}[theorem]{Proposition}
\newtheorem{lemma}[theorem]{Lemma}
\newtheorem{corollary}[theorem]{Corollary}
\theoremstyle{remark}
\newtheorem{remark}[theorem]{Remark}

\title[Extreme least singular values beyond Gaussian entries]{Extreme least singular values of random row submatrices\\
with bounded-density subgaussian entries}

\author{Xiufan Yang}
\address{College of Science, Nanjing University of Posts and Telecommunications, Nanjing, Jiangsu, China}
\email{yxf3173616612@gmail.com}

\author{Shu Wen}
\address{International Institute of Finance, School of Management, University of Science and
Technology of China, Hefei, Anhui, China}
\email{wenshu433@gmail.com}

\author{\protect\rev{Yitzchak Shmalo}}
\address{\rev{Einstein Institute of Mathematics, The Hebrew University of Jerusalem, Givat Ram, Jerusalem, Israel}}
\email{\rev{yitzchak.shmalo@gmail.com}}
\thanks{Corresponding author: Yitzchak Shmalo.}


\begin{document}

\begin{abstract}
\begin{revision}
Let $\xi$ be a centered real subgaussian random variable with positive variance and a bounded
Lebesgue density, and let $A_m\in\R^{N_m\times m}$ have independent entries distributed as
$\xi$, where $N_m/m\to\gamma>1$. Writing
\[
M_m(A_m):=\min_{I\subset[N_m],\,|I|=m}\sigma_{\min}((A_m)_I),
\]
where $(A_m)_I$ denotes the row submatrix indexed by $I$, we determine its exponential scale:
\[
\frac1m\log M_m(A_m)\xrightarrow{\Pp}
-\bigl(\gamma\log\gamma-(\gamma-1)\log(\gamma-1)\bigr).
\]
This extends the corresponding real Gaussian result in \cite{ShmaloGaussian}. The main new
point is an upper-tail argument that avoids uniform control of exponentially many random
hyperplanes. We combine a density-level local central limit theorem for delocalized directions,
an averaged delocalization estimate for hyperplane normals, an exponential bound for nearly
parallel pairs, and amplification by a linear number of independent probe rows. For each fixed
$\varepsilon\in(0,h(\gamma))$, the probability of an $\varepsilon$-deviation is at most
$Ce^{-c\sqrt m}$ for all sufficiently large $m$. Under the canonical coupling by one infinite
i.i.d.\ array, this summable deviation estimate yields a uniform almost-sure exponential law over
every compact range of aspect ratios. In particular, at the real phase-retrieval threshold
$N_m=2m-1$, the Balan--Wang stability parameter has exponential base $1/4$ in probability and,
under this coupling, almost surely.
\end{revision}
\end{abstract}

\maketitle

\section{Introduction}

Let $A\in\R^{N\times m}$ be a tall matrix. This paper studies the most ill-conditioned square row
submatrix of $A$, measured by
\begin{equation}\label{eq:def-M}
M_m(A):=\min_{I\subset[N],\,|I|=m}\sigma_{\min}(A_I),
\end{equation}
where $A_I$ is formed from the rows indexed by $I$. For background on least singular values of
\rev{i.i.d.} random matrices, see \cite{RVrectangular,TaoVu}. Even when the entries of $A$ are
independent, the family in \eqref{eq:def-M} contains exponentially many strongly dependent
matrices. The problem is therefore not reduced to a one-matrix estimate. Related exponential
conditioning questions occur in smoothed analyses of random polytopes and Frank--Wolfe methods;
see \cite{RademacherShu}.

For $\gamma>1$, set
\begin{equation}\label{eq:def-h}
h(\gamma):=\gamma\log\gamma-(\gamma-1)\log(\gamma-1).
\end{equation}
If $N_m/m\to\gamma$, then
\[
\binom{N_m}{m}=\exp\{h(\gamma)m+o(m)\}.
\]
The real Gaussian case was proved in \cite{ShmaloGaussian}:
\[
\frac1m\log M_m(A_m)\xrightarrow{\Pp}-h(\gamma),
\]
the same work obtained the complex Gaussian exponent $-h(\gamma)/2$. The lower-tail half of the real
argument extends to any bounded-density ensemble. The matching non-Gaussian upper tail was left
open in \cite[Remark~6.3]{ShmaloGaussian}, where such an extension was anticipated and the
remaining obstacle was identified as making the local and two-coordinate estimates uniform over
the exponentially many relevant random hyperplanes.

Our main theorem settles that real-valued extension.

\begin{theorem}[Exponential rate beyond Gaussian entries]\label{thm:main}
\begin{revision}
Let $\xi$ be a centered real subgaussian random variable with
\[
\sigma^2:=\operatorname{Var}(\xi)>0.
\]
Suppose that $\xi$ has a Lebesgue density $f_\xi$ with bounded essential supremum,
\[
K_0:=\|f_\xi\|_\infty<\infty.
\]
\end{revision}
Let $A_m\in\R^{N_m\times m}$ have independent entries distributed as $\xi$, and assume that
$N_m/m\to\gamma>1$. Then
\begin{equation}\label{eq:main-convergence}
\frac1m\log M_m(A_m)\xrightarrow{\Pp}-h(\gamma).
\end{equation}
Moreover, for every fixed $\varepsilon\in(0,h(\gamma))$ there are constants $C,c>0$, depending on
$\gamma$, $\varepsilon$, and the law of $\xi$, such that
\begin{equation}\label{eq:quantitative-main}
\Pp\left\{\left|\frac1m\log M_m(A_m)+h(\gamma)\right|>\varepsilon\right\}
\le Ce^{-c\sqrt m}
\end{equation}
for all sufficiently large $m$.
\end{theorem}

\begin{revision}
\begin{remark}[Normalization reduction]\label{rem:normalization}
Let $\sigma^2=\operatorname{Var}(\xi)>0$ and define
\[
\widehat\xi:=\frac{\xi}{\sigma},
\qquad
\widehat A_m:=\frac{1}{\sigma}A_m.
\]
Then $\widehat\xi$ is centered, variance one, and subgaussian. Its density is
\[
f_{\widehat\xi}(x)=\sigma f_\xi(\sigma x),
\qquad
\widehat K:=\|f_{\widehat\xi}\|_\infty=\sigma K_0<\infty.
\]
For every $I\subset[N_m]$ with $|I|=m$,
\[
(A_m)_I=\sigma(\widehat A_m)_I,
\qquad
\sigma_{\min}((A_m)_I)
=\sigma\,\sigma_{\min}((\widehat A_m)_I),
\]
and hence
\begin{equation}\label{eq:normalization-homogeneity}
M_m(A_m)=\sigma M_m(\widehat A_m).
\end{equation}
Consequently,
\begin{equation}\label{eq:normalization-log}
\frac1m\log M_m(A_m)
=\frac{\log\sigma}{m}+\frac1m\log M_m(\widehat A_m).
\end{equation}
Since $(\log\sigma)/m\to0$, the convergence in
\eqref{eq:main-convergence} for $A_m$ follows from the unit-variance statement for
$\widehat A_m$. The quantitative estimate is preserved as well. Indeed, for every fixed
$\varepsilon\in(0,h(\gamma))$ and all sufficiently large $m$,
$|\log\sigma|/m\le\varepsilon/2$, and therefore
\[
\begin{aligned}
&\Pp\left\{\left|\frac1m\log M_m(A_m)+h(\gamma)\right|>\varepsilon\right\}\\
&\qquad\le
\Pp\left\{\left|\frac1m\log M_m(\widehat A_m)+h(\gamma)\right|>
\frac\varepsilon2\right\}
\le Ce^{-c\sqrt m},
\end{aligned}
\]
where the last inequality is the unit-variance estimate applied to the law of
$\widehat\xi$. The constants depend on that normalized law, equivalently on the original law of
$\xi$. Thus it suffices to prove \thmref{thm:main} under the normalization
$\operatorname{Var}(\xi)=1$.
\end{remark}
\end{revision}

\paragraph{Comparison with the Gaussian argument.}
The Gaussian proof uses rotational invariance to identify the hyperplane normals with
Haar-distributed directions. That input is unavailable for general \rev{i.i.d.} entries. We replace
it by four weaker statements that are sufficient when used together.
\begin{enumerate}[label=(\roman*),leftmargin=2.3em]
\item A blocking argument and the density-level Berry--Esseen theorem of Bobkov and G\"otze
\cite{BobkovGoetze} imply a uniform lower bound
\[
\Pp\{|\langle u,X\rangle|\le t\}\ge ct
\]
for every sufficiently delocalized unit vector $u$ and every $0<t\le1$.
\item The normal-vector theorem of Nguyen and Vu \cite{NguyenVu}, used only in expectation,
shows that at least half of the random hyperplane normals are delocalized with high probability.
No union bound over all normals is taken.
\item Pairs of normals that are too close in projective distance occupy an exponentially small
fraction of all pairs. Highly overlapping row sets are handled by entropy, while low-overlap pairs
expose linearly many fresh rows and hence linearly many independent small-ball constraints.
\item A single independent probe row hits one of the good hyperplanes with probability at least of
order $m^{-1/2}$. Reserving $\Theta(m)$ independent probes amplifies this to probability
$1-e^{-\Theta(\sqrt m)}$.
\end{enumerate}
The first three ingredients control an averaged geometry rather than every individual hyperplane.
The fourth converts that averaged statement into the desired extremal conclusion. To the best of
our knowledge, the resulting matching upper bound was not previously available for non-Gaussian
bounded-density subgaussian entries. The complex non-Gaussian analogue remains open.

\begin{revision}
\subsection{Phase-retrieval background and the role of
\texorpdfstring{$\omega(A)$}{omega(A)}}\label{subsec:phase-background}

Given measurement vectors $a_1,\dots,a_N\in\R^m$, phase retrieval asks whether an unknown
signal $x\in\R^m$ can be reconstructed from the magnitudes
$|\langle a_i,x\rangle|$. Writing the $a_i^{\mathsf T}$ as the rows of
$A\in\R^{N\times m}$, the nonlinear measurement map is
\[
\mathcal A_A:\R^m/\{\pm1\}\longrightarrow\R_{\ge0}^N,
\qquad \mathcal A_A([x])=|Ax|,
\]
where the absolute value is taken coordinatewise. The quotient by $\{\pm1\}$ is unavoidable,
because $x$ and $-x$ have identical measurements. The natural distance on the quotient is
\[
d([x],[y]):=\min\{\|x-y\|_2,\|x+y\|_2\}.
\]
The finite-dimensional frame formulation was developed in
\cite{BalanCasazzaEdidin}; a broad account of uniqueness, stability, and the connection with
classical Fourier phase retrieval is given in \cite{GrohsKoppensteinerRathmair}. Convex lifting and
nonconvex spectral-gradient methods provide important algorithmic recovery guarantees for random
measurements \cite{PhaseLift,WirtingerFlow}, but the present paper addresses a different, purely
geometric question: the worst-case conditioning of the phaseless measurement map at the minimum
real redundancy.

In the real case, $\mathcal A_A$ is injective precisely when the rows of $A$ satisfy the
\emph{complement property}: for every $J\subset[N]$, either $A_J$ or $A_{J^c}$ has rank $m$
\cite{BalanCasazzaEdidin,BandeiraEtAl}. Consequently, real phase retrieval requires
$N\ge2m-1$. At the critical value $N=2m-1$, the complement property is equivalent to
\emph{full spark}, meaning that every $m$ rows are linearly independent. Thus a random matrix
with an absolutely continuous entry law is phase retrievable almost surely at this threshold.

Injectivity alone does not quantify robustness. Following Balan and Wang \cite{BalanWang}, define
for a phase-retrievable matrix $A$
\[
q_\varepsilon(A):=
\sup_{\|x\|_2=1}\ \sup_{\substack{y\in\R^m\\
\||Ax|-|Ay|\|_2\le\varepsilon}}
\frac{d([x],[y])}{\varepsilon},
\qquad
q_0(A):=\limsup_{\varepsilon\downarrow0}q_\varepsilon(A).
\]
For $J\subset[N]$, let $A_J$ denote the corresponding row submatrix. The Balan--Wang parameter is
\begin{equation}\label{eq:omega}
\omega(A):=
\min_{\substack{J\subset[N]\\ \rank(A_J)<m}}
\sigma_m(A_{J^c}),
\end{equation}
where $\sigma_m$ is the least singular value of a matrix having at least $m$ rows. Balan and Wang
proved that
\[
q_0(A)=\frac1{\omega(A)}
\]
for every real phase-retrievable frame \cite[Theorem~4.2]{BalanWang}. Hence $\omega(A)$ is the
reciprocal of the sharp infinitesimal worst-case conditioning factor: an exponentially small value
of $\omega(A)$ means exponentially poor worst-case local stability. At the critical threshold,
\propref{prop:full-spark-reduction} reduces this stability parameter exactly to the extreme
least-singular-value statistic in \eqref{eq:def-M}.
\end{revision}

\begin{proposition}[Full-spark reduction]\label{prop:full-spark-reduction}
If $A\in\R^{(2m-1)\times m}$ is full spark, then
\[
\omega(A)=\min_{\substack{T\subset[2m-1]\\|T|=m}}\sigma_{\min}(A_T)=M_m(A).
\]
\end{proposition}

\begin{revision}
\begin{proof}
For a full-spark matrix, $\rank(A_J)<m$ holds exactly when $|J|\le m-1$; hence every admissible
complement $J^c$ contains at least $m$ rows. If $T\subset J^c$ and $|T|=m$, then
\[
A_{J^c}^{\mathsf T}A_{J^c}
=A_T^{\mathsf T}A_T+A_{J^c\setminus T}^{\mathsf T}A_{J^c\setminus T}
\succeq A_T^{\mathsf T}A_T,
\]
so $\sigma_m(A_{J^c})\ge\sigma_{\min}(A_T)$. Taking the minimum over admissible $J$ gives
$\omega(A)\ge M_m(A)$. Conversely, for every $m$-set $T$, the complement $J=T^c$ has size
$m-1$ and is rank deficient, while $J^c=T$. Thus $\omega(A)\le\sigma_{\min}(A_T)$; minimizing
over $T$ proves the reverse inequality.
\end{proof}
\end{revision}

Combining \propref{prop:full-spark-reduction} with \thmref{thm:main} gives the
following non-Gaussian extension of the critical exponential base.

\begin{corollary}[Balan--Wang base for bounded-density subgaussian frames]\label{cor:phase-prob}
Let $A_m\in\R^{(2m-1)\times m}$ have independent entries distributed as a random variable $\xi$
satisfying the assumptions of \thmref{thm:main}. Then
\[
\frac1m\log\omega(A_m)\xrightarrow{\Pp}-\log4,
\qquad
\omega(A_m)=4^{-m+o_{\Pp}(m)}.
\]
Let $R_m:=\max_{1\le i\le2m-1}\|a_i\|_2$, where $a_i$ are the rows of $A_m$. Then, for every
$b>1/4$,
\[
\Pp\{\omega(A_m)\le R_m b^m\}\longrightarrow1,
\]
whereas for every $0<b<1/4$ and every fixed $C>0$,
\[
\Pp\{\omega(A_m)\le C R_m b^m\}\longrightarrow0.
\]
\end{corollary}

\begin{revision}
\begin{corollary}[Uniform almost-sure law under the canonical coupling]
\label{cor:almost-sure}
Let $(\xi_{ij})_{i,j\ge1}$ be an infinite array of independent copies of $\xi$. For $\gamma>1$, set
\[
\begin{aligned}
N_m(\gamma)&:=\lceil\gamma m\rceil-1,\\
A_m(\gamma)&:=(\xi_{ij})_{1\le i\le N_m(\gamma),\,1\le j\le m},\\
\mathcal M_m(\gamma)&:=M_m(A_m(\gamma)).
\end{aligned}
\]
Then, for every compact interval $K\subset(1,\infty)$,
\[
\sup_{\gamma\in K}
\left|\frac1m\log\mathcal M_m(\gamma)+h(\gamma)\right|
\longrightarrow0
\qquad\text{almost surely}.
\]
In particular, the exponential law holds simultaneously for every $\gamma>1$ on one
full-probability event.

At the critical real phase-retrieval threshold, $A_m(2)\in\R^{(2m-1)\times m}$. If
$R_m:=\max_{1\le i\le2m-1}\|(\xi_{ij})_{j=1}^m\|_2$, then almost surely
\[
\omega(A_m(2))^{1/m}\longrightarrow\frac14,
\qquad
R_m^{1/m}\longrightarrow1.
\]
Consequently, for every $b>1/4$, the inequality $\omega(A_m(2))\le R_m b^m$ holds for all
sufficiently large $m$ almost surely; for every $0<b<1/4$ and every fixed $C>0$, the reverse
inequality $\omega(A_m(2))>C R_m b^m$ holds for all sufficiently large $m$ almost surely.
\end{corollary}
\end{revision}

\begin{remark}[Scope of the assumptions]\label{rem:scope}
The assumptions in \thmref{thm:main} are sufficient and are not asserted to be optimal.
Absolute continuity ensures that all square row submatrices and all $(m-1)$-row spans have the
expected rank almost surely, so the logarithm and the normal lines used below are well defined. The
bounded-density hypothesis supplies the uniform one- and two-dimensional anti-concentration
estimates and the density-level local approximation. Subgaussianity supplies exponentially strong
row-norm tails and the available normal-vector delocalization theorem. We do not claim that the
conclusion fails for atomic laws; discrete ensembles lie outside the present method and would
require different anti-concentration and normal-vector inputs.
\end{remark}

\rev{The proof is organized as follows. \secref{sec:preliminaries} establishes the local
small-ball lower bound. \secref{sec:geometry} controls the number and pair geometry of useful
hyperplanes. \secrefs{sec:lower-tail}{sec:upper-tail} prove the two deviations around
the rate $h(\gamma)$. \secref{sec:conclusion} derives the main theorem, its phase-retrieval
consequence, and the almost-sure strengthening.}

\section{Preliminaries and local anti-concentration}\label{sec:preliminaries}

\begin{revision}
By the normalization reduction in \remref{rem:normalization}, throughout the proof sections
we assume $\Ee\xi=0$ and $\operatorname{Var}(\xi)=1$. To avoid additional notation, we relabel
$\widehat\xi$ and $\widehat A_m$ as $\xi$ and $A_m$, and write
\[
K:=\|f_\xi\|_\infty
\]
for the density bound of the normalized law.
\end{revision}

\rev{All logarithms are natural. For a real random variable $Y$, let $\dens(Y)$ denote its
Lebesgue density and let $\|\dens(Y)\|_\infty$ denote the essential supremum of that density. We
use the Orlicz convention}
\[
\rev{\|Y\|_{\psi_2}:=\inf\bigl\{s>0:\Ee\exp(Y^2/s^2)\le2\bigr\},}
\]
\rev{and say that $Y$ is subgaussian when $\|Y\|_{\psi_2}<\infty$.}

Throughout the proof, constants denoted by $c,C,C_1,\dots$ may change from line to line.
Constants in distributional lemmas depend only on the law of $\xi$; constants in tail estimates may
also depend on fixed parameters such as $\gamma$ and $\varepsilon$, but never on $m$.

For unit vectors $u,v\in\R^m$, define their projective distance by
\begin{equation}\label{eq:projective-distance}
\dpr(u,v):=\min\{\|u-v\|_2,\|u+v\|_2\}.
\end{equation}

\begin{revision}
We first record the projection-density theorem of Rudelson and Vershynin
\cite[Theorems~1.1 and 1.2]{RudelsonVershyninDensity}. The distinction between its general
$d$-dimensional form and its sharp one-dimensional form is important for the two-dimensional
estimate below.

\begin{theorem}[Densities of projections]\label{thm:projection-density}
Let $X=(X_1,\dots,X_n)$ have independent real-valued coordinates whose densities are bounded by
$K$ almost everywhere. If $P$ is the orthogonal projection in $\R^n$ onto a $d$-dimensional
subspace, then the density of $PX$, with respect to $d$-dimensional Lebesgue measure on that
subspace, is bounded by $(C_{\mathrm{RV}}K)^d$ almost everywhere, where
$C_{\mathrm{RV}}>0$ is an absolute constant. In dimension $d=1$, the sharper bound
$\sqrt2K$ holds.
\end{theorem}

\begin{lemma}[One- and two-dimensional marginal bounds]\label{lem:marginal-density}
Let $\xi$ be as in \thmref{thm:main}, and let $X=(\xi_1,\dots,\xi_m)$ have independent
coordinates distributed as $\xi$. Then, for every deterministic unit vector $u\in\R^m$,
\begin{equation}\label{eq:one-dimensional-density}
\|\dens(\langle X,u\rangle)\|_\infty\le\sqrt2K.
\end{equation}
Furthermore, for unit vectors $u,v\in\R^m$ with $\rho=\langle u,v\rangle$ and $|\rho|<1$,
\begin{equation}\label{eq:two-dimensional-density}
\|\dens(\langle X,u\rangle,\langle X,v\rangle)\|_\infty
\le \frac{C_{\mathrm{RV}}^2K^2}{\sqrt{1-\rho^2}}.
\end{equation}
\end{lemma}

\begin{proof}
Let $E=\operatorname{span}(u)$ and let $P_E$ be the orthogonal projection onto $E$. Since
$P_EX=\langle X,u\rangle u$, the isometry $t\mapsto tu$ identifies the density of $P_EX$ on $E$
with the density of $\langle X,u\rangle$. The sharp one-dimensional part of
\thmref{thm:projection-density} therefore gives \eqref{eq:one-dimensional-density}.

For the two-dimensional estimate, let $E=\operatorname{span}(u,v)$ and choose the orthonormal
basis
\[
e_1=u,
\qquad
e_2=\frac{v-\rho u}{\sqrt{1-\rho^2}}.
\]
Set $Y=(\langle X,e_1\rangle,\langle X,e_2\rangle)$. The general $d=2$ part of
\thmref{thm:projection-density} gives
\[
\|\dens(Y)\|_\infty\le C_{\mathrm{RV}}^2K^2.
\]
Moreover,
\[
(\langle X,u\rangle,\langle X,v\rangle)=BY,
\qquad
B=\begin{pmatrix}1&0\\ \rho&\sqrt{1-\rho^2}\end{pmatrix},
\]
and $|\det B|=\sqrt{1-\rho^2}$. The change-of-variables formula for densities yields
\[
\|\dens(BY)\|_\infty
\le \frac{\|\dens(Y)\|_\infty}{|\det B|},
\]
which is exactly \eqref{eq:two-dimensional-density}.
\end{proof}
\end{revision}

We next extract the one-dimensional form of the density-level Berry--Esseen theorem needed below.

\begin{theorem}[Density-level Berry--Esseen theorem]\label{thm:density-BE}
Let $Y_1,\dots,Y_r$ be independent centered real random variables such that
\[
\sum_{\ell=1}^r\operatorname{Var}(Y_\ell)=r.
\]
Assume that, for fixed positive constants $b,M,B$, every $\ell$ satisfies
\[
0<\operatorname{Var}(Y_\ell)\le b,
\qquad
\|\dens(Y_\ell)\|_\infty\le M,
\qquad
\Ee|Y_\ell|^3\le B.
\]
If $p_r$ is the density of $r^{-1/2}\sum_{\ell=1}^rY_\ell$ and $\varphi$ is the standard Gaussian
density, then
\begin{equation}\label{eq:density-BE}
\|p_r-\varphi\|_{L^\infty(\R)}\le\frac{C(b,M,B)}{\sqrt r}.
\end{equation}
\end{theorem}

\begin{revision}
\begin{proof}
In dimension one, \cite[Theorem~1, equation~(1.1)]{BobkovGoetze} gives
\[
\|p_r-\varphi\|_{L^\infty(\R)}
\le \frac{C\sigma M^2\beta_3}{\sqrt r},
\]
where $\sigma^2=\max_\ell\operatorname{Var}(Y_\ell)$ and
$\beta_3=r^{-1}\sum_\ell\Ee|Y_\ell|^3$. The assumptions imply
$\sigma\le\sqrt b$ and $\beta_3\le B$, which yields \eqref{eq:density-BE} after absorbing these
fixed quantities into $C(b,M,B)$.
\end{proof}
\end{revision}

\rev{Combining \lemref{lem:marginal-density} with \thmref{thm:density-BE} gives the
local lower bound used in the upper-tail argument. The point is that, when $u$ is delocalized, no
single summand $u_j\xi_j$ contributes appreciably to $\langle u,X\rangle$. The resulting sum lies
in a local central-limit regime, so its density is uniformly positive near the origin.}

\begin{lemma}[Uniform local lower bound]\label{lem:local-lower}
Let $\xi$ and $X$ be as in \lemref{lem:marginal-density}. There are constants $\eta,c_0>0$,
depending only on the law of $\xi$, such that for every $m$, every unit vector $u\in\R^m$ satisfying
$\|u\|_\infty\le\eta$, and every $0<t\le1$,
\begin{equation}\label{eq:local-lower}
\Pp\{|\langle u,X\rangle|\le t\}\ge2c_0t.
\end{equation}
\end{lemma}

\begin{proof}
Fix an integer $q\ge1$, to be chosen sufficiently large in terms of the distribution of $\xi$. We
first assume that
\[
\max_{1\le j\le m}u_j^2\le\frac1{4q}.
\]
Set $w_j:=u_j^2$. Since $\|u\|_2=1$, we have $\sum_{j=1}^mw_j=1$.

We partition the nonzero weights $w_j$ into consecutive blocks by the following greedy procedure.
Starting with an empty block, add nonzero weights until their sum first becomes at least $1/q$, then
close the block and repeat with the remaining weights. Since each weight is at most $1/(4q)$, every
completed block has total weight between $1/q$ and $5/(4q)$. At the end there may remain a
terminal collection of total mass less than $1/q$; merge it into the last completed block.
\rev{Because the total mass is one, at least one block is completed.} We obtain nonempty disjoint
blocks $B_1,\dots,B_r$ whose union contains all $j$ with $u_j\ne0$, and, with
$w_\ell:=\sum_{j\in B_\ell}u_j^2$,
\[
\frac1q\le w_\ell\le\frac9{4q},
\qquad
\sum_{\ell=1}^rw_\ell=1.
\]
The upper bound follows because the last completed block has weight at most $5/(4q)$ before the
terminal remainder, while the remainder has weight less than $1/q$. Consequently,
\[
\frac{4q}{9}\le r\le q.
\]

For $\ell=1,\dots,r$, define
\[
Z_\ell:=\sqrt r\sum_{j\in B_\ell}u_j\xi_j.
\]
The variables $Z_1,\dots,Z_r$ are independent and centered. Since
$\operatorname{Var}(\xi)=1$,
\[
\operatorname{Var}(Z_\ell)=r\sum_{j\in B_\ell}u_j^2=rw_\ell,
\]
and therefore
\[
\frac49\le\operatorname{Var}(Z_\ell)\le\frac94,
\qquad
\sum_{\ell=1}^r\operatorname{Var}(Z_\ell)=r.
\]

We next establish a uniform density bound. If $B_\ell$ contains at least two indices, put
$a_j:=u_j/\sqrt{w_\ell}$ for $j\in B_\ell$. Then $\sum_{j\in B_\ell}a_j^2=1$, so
\lemref{lem:marginal-density} gives
\[
\left\|\dens\left(\sum_{j\in B_\ell}a_j\xi_j\right)\right\|_\infty\le\sqrt2K.
\]
Since
\[
Z_\ell=\sqrt{rw_\ell}\sum_{j\in B_\ell}a_j\xi_j,
\]
the scaling rule for densities yields
\[
\|\dens(Z_\ell)\|_\infty
\le\frac{\sqrt2K}{\sqrt{rw_\ell}}
\le\frac{3\sqrt2}{2}K.
\]
If $B_\ell=\{j\}$ is a singleton, then $Z_\ell=\sqrt r\,u_j\xi_j$, and
\[
\|\dens(Z_\ell)\|_\infty
\le\frac{K}{\sqrt r\,|u_j|}
=\frac{K}{\sqrt{rw_\ell}}
\le\frac32K.
\]
Thus $\|\dens(Z_\ell)\|_\infty\le C_1K$ in all cases.

The third moments are also uniformly bounded. Since $\xi$ is subgaussian and the coordinates are
independent and centered, the standard subgaussian inequality for linear combinations
\cite[Proposition~2.6.1]{VershyninBook} gives
\[
\left\|\sum_{j\in B_\ell}u_j\xi_j\right\|_{\psi_2}
\le C\|\xi\|_{\psi_2}\left(\sum_{j\in B_\ell}u_j^2\right)^{1/2}.
\]
Consequently,
\[
\|Z_\ell\|_{\psi_2}
\le C\|\xi\|_{\psi_2}\sqrt{rw_\ell}
\le\frac{3C}{2}\|\xi\|_{\psi_2},
\]
so $\Ee|Z_\ell|^3\le C_2$, where $C_2$ depends only on the law of $\xi$.

We may therefore apply \thmref{thm:density-BE} with
$b=9/4$, $M=C_1K$, and $B=C_2$. Since
\[
\frac1{\sqrt r}\sum_{\ell=1}^rZ_\ell
=\sum_{j=1}^mu_j\xi_j
=\langle u,X\rangle,
\]
if $p_u$ denotes the density of $\langle u,X\rangle$, then
\[
\|p_u-\varphi\|_{L^\infty(\R)}
\le\frac{C_3}{\sqrt r}
\le\frac{3C_3}{2\sqrt q}.
\]
Choose $q$ sufficiently large that $3C_3/(2\sqrt q)\le\varphi(1)/2$. Then, for almost every
$|x|\le1$,
\[
p_u(x)\ge\varphi(x)-\frac{\varphi(1)}2\ge\frac{\varphi(1)}2.
\]
Set
\[
\eta:=\frac1{2\sqrt q},
\qquad
c_0:=\frac{\varphi(1)}2.
\]
If $\|u\|_\infty\le\eta$, then $\max_ju_j^2\le1/(4q)$, so all preceding estimates apply. Finally,
for every $0<t\le1$,
\[
\Pp\{|\langle u,X\rangle|\le t\}
=\int_{-t}^tp_u(x)\,dx
\ge\int_{-t}^tc_0\,dx
=2c_0t.
\]
\end{proof}

\section{Geometry of random hyperplane normals}\label{sec:geometry}

Because the entry law is absolutely continuous, with probability one every row family of size at
most $m$ that appears below has the expected rank. We work on this full-probability event. To make
the normal-vector selection explicit, for an $(m-1)$-set $S$ let $B_S$ be the corresponding row
matrix, with its rows ordered by increasing index. If $B_S^{(j)}$ is obtained by deleting column
$j$, set
\[
c_S:=\bigl((-1)^{j-1}\det B_S^{(j)}\bigr)_{j=1}^m,
\qquad
u_S:=\frac{c_S}{\|c_S\|_2}.
\]
Then $c_S\ne0$, $B_Sc_S=0$, and $u_S$ is a measurable unit normal. All events used below are
unchanged if a normal is replaced by its negative.

The following is the largest-coordinate part of the normal-vector theorem of Nguyen and Vu
\cite[Theorem~1.4, equation~(6)]{NguyenVu}.

\begin{theorem}[Nguyen--Vu]\label{thm:NV}
Let $B$ be an $(m-1)\times m$ matrix with independent entries distributed as a centered,
variance-one subgaussian random variable, and let $u$ be a unit normal to its rows. There are
constants $C,C_1>0$, depending only on the entry law, such that, for all sufficiently large $m$
and every integer $s\ge C_1\log m$,
\begin{equation}\label{eq:NV}
\Pp\left\{\|u\|_\infty\ge\sqrt{\frac{s}{m}}\right\}
\le Cm^2e^{-s/C}.
\end{equation}
\end{theorem}

\thmref{thm:NV} provides a delocalization estimate for the unit normal of a single random
hyperplane. The following lemma shows that, with high probability, at least half of the hyperplanes
generated by the random rows have delocalized normals.

\begin{lemma}[Many hyperplanes have delocalized \rev{normals}]\label{lem:many-good}
Let $n=n_m\ge m$, let $R_1,\dots,R_n$ be independent rows with \rev{i.i.d.} coordinates
distributed as $\xi$, and for every $S\subset[n]$ with $|S|=m-1$ let $u_S$ be a unit normal to the
rows indexed by $S$. Put
\[
\cS:=\{S\subset[n]:|S|=m-1\},
\qquad
L:=|\cS|=\binom{n}{m-1},
\]
and, with $\eta$ from \lemref{lem:local-lower}, define the set of \rev{good hyperplanes}
\[
\cG:=\{S\in\cS:\|u_S\|_\infty\le\eta\}.
\]
Then there are constants $C,c>0$ such that
\begin{equation}\label{eq:many-good}
\Pp\{|\cG|<L/2\}\le Cm^2e^{-cm}
\end{equation}
for all sufficiently large $m$.
\end{lemma}

\begin{proof}
Fix $S\in\cS$. The matrix whose rows are indexed by $S$ is an $(m-1)\times m$ random matrix
with independent entries distributed as $\xi$, so \thmref{thm:NV} applies to $u_S$. Let
$s_m:=\lfloor\eta^2m\rfloor$. \rev{For all sufficiently large $m$, one has
$s_m\ge C_1\log m$ and $\sqrt{s_m/m}\le\eta$.} Consequently,
\[
\{S\notin\cG\}=\{\|u_S\|_\infty>\eta\}
\subseteq
\left\{\|u_S\|_\infty\ge\sqrt{\frac{s_m}{m}}\right\}.
\]
Applying \thmref{thm:NV} gives
\[
\Pp\{S\notin\cG\}\le Cm^2e^{-s_m/C}.
\]
For all sufficiently large $m$, $s_m\ge\eta^2m/2$, and hence
\[
\Pp\{S\notin\cG\}\le Cm^2e^{-cm}.
\]
Let
\[
N_{\mathrm{bad}}:=|\cS\setminus\cG|
=\sum_{S\in\cS}\1_{\{S\notin\cG\}}.
\]
By linearity of expectation,
\[
\Ee N_{\mathrm{bad}}
=\sum_{S\in\cS}\Pp\{S\notin\cG\}
\le LCm^2e^{-cm}.
\]
Since $|\cG|<L/2$ implies $N_{\mathrm{bad}}>L/2$, Markov's inequality yields
\[
\Pp\{|\cG|<L/2\}
\le\frac{2\Ee N_{\mathrm{bad}}}{L}
\le2Cm^2e^{-cm}.
\]
Absorbing the factor $2$ into $C$ proves the claim.
\end{proof}

We next control the geometry of pairs. The following entropy estimate shows that pairs of row sets
with very large overlap form an exponentially negligible fraction.

\begin{lemma}[Binomial entropy]\label{lem:entropy}
Suppose $n_m/m\to\beta>1$ and let
\[
L_m:=\binom{n_m}{m-1}.
\]
Then
\begin{equation}\label{eq:L-entropy}
\frac1m\log L_m\longrightarrow h(\beta).
\end{equation}
Moreover, for every $\alpha\in(0,h(\beta))$ there are constants $\delta_0,\kappa>0$ such that, for all
sufficiently large $m$,
\begin{equation}\label{eq:overlap-entropy}
\frac1{L_m}
\sum_{1\le k\le\lfloor\delta_0m\rfloor}
\binom{m-1}{k}\binom{n_m-m+1}{k}
\le e^{-(\alpha+4\kappa)m}.
\end{equation}
\end{lemma}

\begin{proof}
Write $\beta_m:=n_m/m$, so $\beta_m\to\beta$. Stirling's formula gives
\[
\begin{aligned}
\frac1m\log L_m
&=\frac1m\bigl(\log(n_m!)-\log((m-1)!)-\log((n_m-m+1)!)\bigr)\\
&=\beta_m\log\beta_m-(\beta_m-1)\log(\beta_m-1)+o(1),
\end{aligned}
\]
which proves \eqref{eq:L-entropy}.

Fix $\alpha<h(\beta)$ and choose $\kappa>0$ so small that
$\alpha+7\kappa<h(\beta)$. By \eqref{eq:L-entropy}, for all sufficiently large $m$,
\[
L_m\ge e^{(h(\beta)-\kappa)m}.
\]
Since $n_m/m\to\beta$, there is $A>0$ such that $n_m-m+1\le Am$ for all sufficiently large $m$.
Using $\binom pk\le(ep/k)^k$, we obtain, for $1\le k\le\delta_0m$ and $x:=k/m$,
\[
\frac1m\log\left[\binom{m-1}{k}\binom{n_m-m+1}{k}\right]
\le x\log\left(\frac{Ae^2}{x^2}\right).
\]
The right-hand side tends to zero as $x\downarrow0$. \rev{Choose $\delta_0>0$ sufficiently small
that $\delta_0<\frac12\min\{1,\beta-1\}$ and}
\[
\rev{x\log\left(\frac{Ae^2}{x^2}\right)\le\frac\kappa2
\qquad(0<x\le\delta_0).}
\]
Then every summand in \eqref{eq:overlap-entropy} is at most $e^{\kappa m/2}$, and the sum contains
at most $m$ terms. Hence, for all sufficiently large $m$,
\[
\sum_{1\le k\le\lfloor\delta_0m\rfloor}
\binom{m-1}{k}\binom{n_m-m+1}{k}
\le me^{\kappa m/2}\le e^{\kappa m}.
\]
Dividing by the lower bound for $L_m$ gives
\[
\frac1{L_m}
\sum_{1\le k\le\lfloor\delta_0m\rfloor}
\binom{m-1}{k}\binom{n_m-m+1}{k}
\le e^{-(h(\beta)-2\kappa)m}.
\]
Since $h(\beta)-2\kappa>\alpha+4\kappa$, this proves \eqref{eq:overlap-entropy}.
\end{proof}

\begin{lemma}[Row-norm tail]\label{lem:row-norm}
For every $A>0$ there is $R=R(A,\xi)$ such that, for
$X=(\xi_1,\dots,\xi_m)$,
\begin{equation}\label{eq:row-norm-tail}
\Pp\{\|X\|_2>R\sqrt m\}\le e^{-Am}
\end{equation}
for all sufficiently large $m$.
\end{lemma}

\begin{proof}
Since $\xi$ is subgaussian, there is $\lambda>0$, depending only on its law, such that
$D:=\Ee e^{\lambda\xi^2}<\infty$. Markov's inequality gives
\[
\begin{aligned}
\Pp\left\{\sum_{j=1}^m\xi_j^2>R^2m\right\}
&\le e^{-\lambda R^2m}\Ee\exp\left(\lambda\sum_{j=1}^m\xi_j^2\right)\\
&=e^{-\lambda R^2m}\prod_{j=1}^m\Ee e^{\lambda\xi_j^2}\\
&=\exp\bigl\{-(\lambda R^2-\log D)m\bigr\}.
\end{aligned}
\]
\rev{Because $\{\|X\|_2>R\sqrt m\}=\{\sum_{j=1}^m\xi_j^2>R^2m\}$, choosing $R$ so
large that $\lambda R^2-\log D\ge A$ proves the claim.}
\end{proof}

\begin{lemma}[Few nearly parallel pairs]\label{lem:few-parallel}
Suppose $n_m/m\to\beta>1$. With the notation above, fix $\alpha\in(0,h(\beta))$. There are constants
$\delta_0,a,\kappa,c_*>0$ such that, with probability at least $1-e^{-c_*m}$ for all sufficiently
large $m$, the number of ordered distinct pairs $(S,T)\in\cS^2$ satisfying at least one of
\begin{equation}\label{eq:exceptional-pairs}
|T\setminus S|\le\delta_0m,
\qquad\text{or}\qquad
\dpr(u_S,u_T)<\frac{a}{\sqrt m},
\end{equation}
is at most
\[
2L_m^2e^{-(\alpha+4\kappa)m}.
\]
\end{lemma}

\begin{proof}
Fix $\alpha<h(\beta)$. By \lemref{lem:entropy}, there are $\delta_0,\kappa>0$ such that
\[
\frac1{L_m}
\sum_{1\le k\le\lfloor\delta_0m\rfloor}
\binom{m-1}{k}\binom{n_m-m+1}{k}
\le e^{-(\alpha+4\kappa)m}
\]
for all sufficiently large $m$.

We first count pairs with large overlap. Fix $S\in\cS$ and let $k:=|T\setminus S|$. Since
$|S|=|T|=m-1$, one also has $|S\setminus T|=k$. To construct $T$, remove $k$ indices from $S$
and add $k$ indices from its complement. Thus the number of such $T$ is exactly
\[
\binom{m-1}{k}\binom{n_m-m+1}{k}.
\]
Summing over $S$ and $1\le k\le\lfloor\delta_0m\rfloor$, the number of ordered distinct pairs with
$|T\setminus S|\le\delta_0m$ is at most
\[
L_m^2e^{-(\alpha+4\kappa)m}.
\]
This bound is deterministic.

It remains to control pairs satisfying $k:=|T\setminus S|>\delta_0m$ and
$\dpr(u_S,u_T)<a/\sqrt m$. Fix such an ordered pair and condition on all rows indexed by $S$.
Then $u:=u_S$ is fixed, while the rows $R_i$, $i\in T\setminus S$, remain independent of $u$ and
of one another. If $\dpr(u_S,u_T)<a/\sqrt m$, choose $v\in\{u_T,-u_T\}$ so that
$\|u-v\|_2<a/\sqrt m$. For every $i\in T\setminus S$, one has $\langle R_i,v\rangle=0$.
Introduce
\[
\mathcal N:=\left\{\max_{i\in T\setminus S}\|R_i\|_2>R\sqrt m\right\}.
\]
On $\mathcal N^c$,
\[
|\langle R_i,u\rangle|
=|\langle R_i,u-v\rangle|
\le\|R_i\|_2\|u-v\|_2
<Ra.
\]
Consequently,
\[
\left\{\dpr(u_S,u_T)<\frac a{\sqrt m}\right\}
\subseteq
\mathcal N\cup\bigcap_{i\in T\setminus S}
\{|\langle R_i,u\rangle|\le Ra\}.
\]

All probabilities in the following estimates are conditional on the rows indexed by $S$; the
bounds are uniform in their realization and therefore remain valid after averaging. Choose
$A>\alpha+8\kappa$. \lemref{lem:row-norm}, applied with exponent $A+2$, gives an
$R>0$ such that
\[
\Pp\{\|R_i\|_2>R\sqrt m\}\le e^{-(A+2)m}.
\]
Since $|T\setminus S|\le n_m=O(m)$,
\[
\Pp(\mathcal N)\le e^{-(A+1)m}
\]
for all sufficiently large $m$.

Conditional on $u$, \lemref{lem:marginal-density} gives
\[
\Pp\{|\langle R_i,u\rangle|\le Ra\mid u\}
\le2\sqrt2KRa.
\]
The fresh rows are conditionally independent, so
\[
\Pp\left\{\bigcap_{i\in T\setminus S}
\{|\langle R_i,u\rangle|\le Ra\}\,\middle|\,u\right\}
\le(2\sqrt2KRa)^k.
\]
Choose $a>0$ so small that
\[
2\sqrt2KRa\le e^{-(A+1)/\delta_0}<1.
\]
Since $k>\delta_0m$,
\[
(2\sqrt2KRa)^k\le e^{-(A+1)m}.
\]
Thus, uniformly over every fixed ordered pair with $|T\setminus S|>\delta_0m$,
\[
\Pp\left\{\dpr(u_S,u_T)<\frac a{\sqrt m}\right\}
\le2e^{-(A+1)m}\le2e^{-Am}.
\]

\rev{Let $F_m$ denote the number of ordered pairs with $|T\setminus S|>\delta_0m$ and
$\dpr(u_S,u_T)<a/\sqrt m$. By linearity of expectation,}
\[
\rev{\Ee F_m\le2L_m^2e^{-Am}.}
\]
\rev{Markov's inequality therefore gives}
\[
\rev{\Pp\{F_m>L_m^2e^{-(\alpha+4\kappa)m}\}
\le2e^{-(A-\alpha-4\kappa)m}.}
\]
\rev{Set $c_*:=(A-\alpha-4\kappa)/2>0$. For all sufficiently large $m$, the last expression is at
most $e^{-c_*m}$. Outside this exceptional event, the low-overlap nearly parallel pairs contribute
at most $L_m^2e^{-(\alpha+4\kappa)m}$ pairs. Adding the deterministic high-overlap bound proves
the lemma.}
\end{proof}

\section{Lower-tail estimate}\label{sec:lower-tail}

\begin{proposition}[No submatrix is much smaller]\label{prop:lower-tail}
Under the assumptions of \thmref{thm:main}, for every $\varepsilon>0$,
\begin{equation}\label{eq:lower-tail-event}
\Pp\{M_m(A_m)\le e^{-(h(\gamma)+\varepsilon)m}\}\longrightarrow0.
\end{equation}
\rev{Moreover,} there are constants $C,c>0$, depending on $\gamma$, $\varepsilon$, and the law of
$\xi$, such that the probability in \eqref{eq:lower-tail-event} is at most $Ce^{-cm}$ for all
sufficiently large $m$.
\end{proposition}

\begin{proof}
Let $B$ be \rev{an $m\times m$ matrix with i.i.d. entries distributed as $\xi$}, with rows
$R_1,\dots,R_m$. By absolute continuity, $B$ is invertible almost surely; we work on this
full-probability event. Put
\[
H_i:=\operatorname{span}\{R_j:j\ne i\},
\qquad
d_i:=\dist(R_i,H_i).
\]
The negative second-moment identity gives
\begin{equation}\label{eq:negative-second-moment}
\sum_{j=1}^m\sigma_j(B)^{-2}=\sum_{i=1}^md_i^{-2}.
\end{equation}
Indeed, the $i$th column of $B^{-1}$ has norm $d_i^{-1}$, and
\eqref{eq:negative-second-moment} follows by taking the squared Frobenius norm of $B^{-1}$.

If $\sigma_{\min}(B)\le t$, then \eqref{eq:negative-second-moment} implies
$d_i\le\sqrt m\,t$ for some $i$. Conditional on all rows except $R_i$, let $u_i$ be a unit normal to
$H_i$. Then $d_i=|\langle R_i,u_i\rangle|$, and \eqref{eq:one-dimensional-density} gives
\[
\Pp\{d_i\le\sqrt m\,t\mid(R_j)_{j\ne i}\}
\le2\sqrt2K\sqrt m\,t.
\]
A union bound over $i$ yields
\begin{equation}\label{eq:one-matrix-lower-tail}
\Pp\{\sigma_{\min}(B)\le t\}\le CKm^{3/2}t.
\end{equation}

There are $\binom{N_m}{m}$ square row submatrices. Taking
$t=e^{-(h(\gamma)+\varepsilon)m}$ and applying another union bound gives
\[
\Pp\{M_m(A_m)\le t\}
\le CKm^{3/2}\binom{N_m}{m}t
=\exp\{-\varepsilon m+o(m)\}.
\]
Since $m^{-1}\log\binom{N_m}{m}\to h(\gamma)$, for all sufficiently large $m$,
\[
\binom{N_m}{m}\le e^{(h(\gamma)+\varepsilon/2)m}.
\]
Hence the preceding probability is at most
$CKm^{3/2}e^{-\varepsilon m/2}\le Ce^{-cm}$, proving the quantitative estimate without any rate
assumption on the convergence $N_m/m\to\gamma$.
\end{proof}

\section{Upper-tail estimate}\label{sec:upper-tail}

\begin{proposition}[A submatrix reaches the predicted scale]\label{prop:upper-tail}
Under the assumptions of \thmref{thm:main}, for every $\varepsilon>0$,
\begin{equation}\label{eq:upper-tail-event}
\Pp\{M_m(A_m)>e^{-(h(\gamma)-\varepsilon)m}\}\longrightarrow0.
\end{equation}
\rev{Moreover,} if $0<\varepsilon<h(\gamma)$, then there are constants $C,c>0$, depending on
$\gamma$, $\varepsilon$, and the law of $\xi$, such that the probability in
\eqref{eq:upper-tail-event} is at most $Ce^{-c\sqrt m}$ for all sufficiently large $m$.
\end{proposition}

\begin{proof}
\rev{It suffices to prove the quantitative statement for $0<\varepsilon<h(\gamma)$: if
$\varepsilon\ge h(\gamma)$, the event in \eqref{eq:upper-tail-event} is contained in the
corresponding event for any fixed $\varepsilon_0\in(0,h(\gamma))$.}
Set
\[
\alpha:=h(\gamma)-\varepsilon>0.
\]
Since $h$ is continuous and strictly increasing on $(1,\infty)$, choose
$0<\delta<\gamma-1$ such that $h(\gamma-\delta)>\alpha$. Reserve
$r_m:=\lfloor\delta m\rfloor$ rows of $A_m$ as probes and call the remaining
$n_m:=N_m-r_m$ rows the base rows. Then
\[
\frac{n_m}{m}\longrightarrow\beta:=\gamma-\delta>1,
\qquad
L_m:=\binom{n_m}{m-1}=\exp\{h(\beta)m+o(m)\}.
\]
For every $(m-1)$-subset $S$ of the base rows, let $u_S$ be a unit normal to their span. Let
$\cG$ be the good family from \lemref{lem:many-good}. Apply \lemref{lem:few-parallel}
with the present $\alpha<h(\beta)$, and let $\cE$ be the collection of ordered distinct pairs
$(S,T)$ satisfying at least one of the two conditions in \eqref{eq:exceptional-pairs}. By
\lemref{lem:many-good} and \lemref{lem:few-parallel}, with probability at least
$1-Ce^{-cm}$ over the base rows, for all sufficiently large $m$,
\begin{equation}\label{eq:good-base}
|\cG|\ge\frac{L_m}{2},
\qquad
|\cE|\le2L_m^2e^{-(\alpha+4\kappa)m}.
\end{equation}

Fix a base realization satisfying \eqref{eq:good-base}. Put
$t_m:=e^{-\alpha m}$, and let $X$ be one independent probe row. Define
\[
Z_X:=\sum_{S\in\cG}\1_{\{|\langle X,u_S\rangle|\le t_m\}}.
\]
By \lemref{lem:local-lower},
\begin{equation}\label{eq:mu-lower}
\mu:=\Ee[Z_X\mid\text{base}]\ge2c_0|\cG|t_m.
\end{equation}
Since $h(\beta)>\alpha$,
\begin{equation}\label{eq:Gt-diverges}
|\cG|t_m=\exp\{(h(\beta)-\alpha)m+o(m)\}\longrightarrow\infty.
\end{equation}

We estimate the conditional second moment. By \eqref{eq:one-dimensional-density}, every diagonal
term is at most $Ct_m$. If $(S,T)\notin\cE$, then
\[
\dpr(u_S,u_T)\ge\frac a{\sqrt m}.
\]
Let $\rho:=|\langle u_S,u_T\rangle|$. The definition of projective distance gives the exact identity
\[
1-\rho^2=(1+\rho)(1-\rho)
=\frac{1+\rho}{2}\dpr(u_S,u_T)^2.
\]
Consequently,
\begin{equation}\label{eq:rho-separation}
1-\rho^2\ge\frac12\dpr(u_S,u_T)^2\ge\frac{a^2}{2m}.
\end{equation}
By the corrected two-dimensional density estimate \eqref{eq:two-dimensional-density}, followed by
integration over $[-t_m,t_m]^2$,
\begin{equation}\label{eq:pair-small-ball}
\Pp\bigl\{|\langle X,u_S\rangle|\le t_m,\ |\langle X,u_T\rangle|\le t_m
\mid\text{base}\bigr\}
\le C\sqrt m\,t_m^2.
\end{equation}
For an exceptional pair, the one-dimensional estimate gives the crude bound $Ct_m$. Therefore,
\begin{equation}\label{eq:second-moment}
\Ee[Z_X^2\mid\text{base}]
\le C|\cG|t_m+C\sqrt m\,|\cG|^2t_m^2+C|\cE|t_m.
\end{equation}

\begin{revision}
Dividing \eqref{eq:second-moment} by the square of \eqref{eq:mu-lower} and using
\eqref{eq:good-base}--\eqref{eq:Gt-diverges}, we obtain
\[
\begin{aligned}
\frac{\Ee[Z_X^2\mid\text{base}]}{\mu^2}
&\le \frac{C}{|\cG|t_m}+C\sqrt m+
C\frac{L_m^2e^{-(\alpha+4\kappa)m}}{|\cG|^2t_m}\\
&\le \frac{C}{|\cG|t_m}+C\sqrt m+Ce^{-4\kappa m}\\
&\le C'\sqrt m+o(1).
\end{aligned}
\]
The Paley--Zygmund inequality at level zero (equivalently, Cauchy--Schwarz applied to
$Z_X\1_{\{Z_X>0\}}$) gives, uniformly over every base realization satisfying
\eqref{eq:good-base},
\begin{equation}\label{eq:one-probe-hit}
\Pp\{Z_X>0\mid\text{base}\}
\ge\frac{\mu^2}{\Ee[Z_X^2\mid\text{base}]}
\ge\frac{c_1}{\sqrt m}.
\end{equation}
\end{revision}

Conditional on the base rows, the $r_m$ probe rows are independent and each satisfies
\eqref{eq:one-probe-hit}. Hence the probability that none of them hits a good base hyperplane is at
most
\begin{equation}\label{eq:amplification}
\left(1-\frac{c_1}{\sqrt m}\right)^{r_m}
\le e^{-c_1r_m/\sqrt m}
\le e^{-c_2\sqrt m}.
\end{equation}
If a probe row $X$ hits $S$, form the square matrix $B$ from the rows indexed by $S$ together with
$X$. Since the rows in $S$ annihilate $u_S$,
\[
\sigma_{\min}(B)\le\|Bu_S\|_2=|\langle X,u_S\rangle|\le t_m.
\]
Thus $M_m(A_m)\le t_m$. Combining \eqref{eq:good-base} and \eqref{eq:amplification} gives
\[
\Pp\{M_m(A_m)>t_m\}
\le Ce^{-cm}+e^{-c_2\sqrt m}
\le C'e^{-c'\sqrt m},
\]
which proves \eqref{eq:upper-tail-event} and the stated quantitative bound.
\end{proof}

\section{Conclusion and applications}\label{sec:conclusion}

\begin{proof}[Proof of \thmref{thm:main}]
\rev{By the normalization reduction in \remref{rem:normalization}, it is enough to prove the
result under the standing unit-variance normalization.}
Fix $\varepsilon>0$. \propref{prop:lower-tail} gives
\[
\Pp\left\{\frac1m\log M_m(A_m)<-h(\gamma)-\varepsilon\right\}\longrightarrow0,
\]
and \propref{prop:upper-tail} gives
\[
\Pp\left\{\frac1m\log M_m(A_m)>-h(\gamma)+\varepsilon\right\}\longrightarrow0.
\]
Because the entry law is absolutely continuous, every square row submatrix is invertible almost
surely, and there are finitely many such submatrices for each $m$. Hence $M_m(A_m)>0$ almost
surely and the logarithm is well defined. The two tail estimates prove
\eqref{eq:main-convergence}. For $0<\varepsilon<h(\gamma)$, their quantitative forms yield
\eqref{eq:quantitative-main}.
\end{proof}

\begin{proof}[Proof of \corref{cor:phase-prob}]
\begin{revision}
Let $\sigma^2=\operatorname{Var}(\xi)$ and $\widehat A_m=A_m/\sigma$. Homogeneity gives
\[
\omega(A_m)=\sigma\omega(\widehat A_m),
\qquad
R_m(A_m)=\sigma R_m(\widehat A_m).
\]
Thus the logarithmic limit changes only by $(\log\sigma)/m$, while each comparison of
$\omega(A_m)$ with $R_m(A_m)b^m$ is exactly invariant under the common factor $\sigma$.
It therefore suffices to prove the corollary in the unit-variance normalization.
\end{revision}
Absolute continuity implies that $A_m$ is full spark almost surely. By
\propref{prop:full-spark-reduction},
\[
\omega(A_m)=M_m(A_m).
\]
Since $h(2)=\log4$, \thmref{thm:main} gives the first assertion.

It remains to compare $\omega(A_m)$ with the largest row norm. For every fixed $\eta>0$,
subgaussianity and a union bound give
\[
\Pp\{R_m>e^{\eta m}\}\longrightarrow0.
\]
On the other hand, $R_m\ge|(A_m)_{11}|$ and bounded density gives
\[
\Pp\{R_m<e^{-\eta m}\}
\le\Pp\{|\xi|<e^{-\eta m}\}
\le2Ke^{-\eta m}.
\]
Thus $m^{-1}\log R_m\to0$ in probability.

If $b>1/4$, choose $\varepsilon,\eta>0$ with
$\varepsilon+\eta<\log(4b)$. With probability tending to one,
\[
\omega(A_m)\le e^{(-\log4+\varepsilon)m}
\le e^{-\eta m}b^m
\le R_mb^m.
\]
If $0<b<1/4$, choose $\varepsilon,\eta>0$ with
$\varepsilon+\eta<-\log(4b)$. With probability tending to one,
\[
\omega(A_m)\ge e^{(-\log4-\varepsilon)m},
\qquad
R_m\le e^{\eta m}.
\]
For every fixed $C>0$,
\[
Ce^{(\eta+\log b)m}<e^{(-\log4-\varepsilon)m}
\]
for all sufficiently large $m$, proving the remaining claim.
\end{proof}

\begin{revision}
\begin{proof}[Proof of \corref{cor:almost-sure}]
Let $\sigma^2=\operatorname{Var}(\xi)$ and put
$\widehat\xi_{ij}:=\xi_{ij}/\sigma$. If $\widehat{\mathcal M}_m(\gamma)$,
$\widehat\omega_m$, and $\widehat R_m$ denote the quantities formed from the normalized array,
then for every $m$ and $\gamma>1$,
\[
\mathcal M_m(\gamma)=\sigma\widehat{\mathcal M}_m(\gamma),
\qquad
\omega(A_m(2))=\sigma\widehat\omega_m,
\qquad
R_m=\sigma\widehat R_m.
\]
Hence the uniform logarithmic expression changes by the deterministic term
$(\log\sigma)/m$, and the two phase-retrieval quantities acquire the same common factor.
It is therefore enough to prove the assertions for the normalized array.
By absolute continuity and countability, with probability one every finite square row submatrix of
the infinite array is nonsingular. We work on this event whenever logarithms are taken. For
$q\in\mathbb Q\cap(1,\infty)$, put
\[
f_m(q):=\frac1m\log\mathcal M_m(q).
\]
The sequence $N_m(q)/m$ converges to $q$. Hence, for every rational
$\varepsilon\in(0,h(q))$, \thmref{thm:main} gives
\[
\sum_{m=1}^\infty
\Pp\{|f_m(q)+h(q)|>\varepsilon\}<\infty,
\]
because $\sum_me^{-c\sqrt m}<\infty$. The first Borel--Cantelli lemma, followed by a countable
intersection over rational pairs $(q,\varepsilon)$, produces a full-probability event on which the
corresponding deviation events occur only finitely often. For any $\delta>0$, one may choose a
rational $\varepsilon\in(0,\min\{\delta,h(q)\})$; hence this event satisfies
\begin{equation}\label{eq:rational-as-convergence}
f_m(q)\longrightarrow-h(q)
\qquad\text{for every }q\in\mathbb Q\cap(1,\infty).
\end{equation}
No independence across dimensions is used in this step.

Fix an outcome in this event and a compact interval $K=[\gamma_0,\gamma_1]\subset(1,\infty)$.
For each $m$, the function
\[
\gamma\longmapsto f_m(\gamma):=\frac1m\log\mathcal M_m(\gamma)
\]
is nonincreasing, because increasing $\gamma$ only adds rows and hence enlarges the family over
which the minimum is taken. Let $\varepsilon>0$. By continuity of $h$, choose rational points
\[
1<q_0<q_1<\cdots<q_L,
\qquad
q_0<\gamma_0<\gamma_1<q_L,
\]
such that every $\gamma\in K$ lies in some $[q_j,q_{j+1}]$ and
\[
h(q_{j+1})-h(q_j)<\frac\varepsilon3
\qquad(0\le j<L).
\]
By \eqref{eq:rational-as-convergence}, for all sufficiently large $m$,
\[
|f_m(q_j)+h(q_j)|<\frac\varepsilon3
\qquad(0\le j\le L).
\]
If $q_j\le\gamma\le q_{j+1}$, monotonicity gives
$f_m(q_{j+1})\le f_m(\gamma)\le f_m(q_j)$. Therefore
\[
-\frac{2\varepsilon}{3}
<f_m(\gamma)+h(\gamma)
<\frac{2\varepsilon}{3}.
\]
This proves the asserted uniform convergence on $K$. The preceding argument is deterministic on
the single full-probability event furnished by \eqref{eq:rational-as-convergence}; hence it applies
to every compact interval in $(1,\infty)$, and therefore to every compact subset of
$(1,\infty)$.

At $\gamma=2$, one has $N_m(2)=2m-1$. Absolute continuity and a countable intersection imply
that all $A_m(2)$ are full spark almost surely. Thus
$\omega(A_m(2))=\mathcal M_m(2)$ for every $m$ on one full-probability event, and the uniform law
gives $\omega(A_m(2))^{1/m}\to1/4$ almost surely.

It remains to control $R_m$. Fix $\eta>0$, choose exponent $3$ in
\lemref{lem:row-norm}, and let $R$ be the resulting constant. For all sufficiently large $m$,
$e^{\eta m}>R\sqrt m$, so a union bound over the $2m-1$ rows gives
\[
\Pp\{R_m>e^{\eta m}\}\le(2m-1)e^{-3m}.
\]
Also, $R_m\ge|\xi_{11}|$, and the bounded-density assumption gives
\[
\Pp\{R_m<e^{-\eta m}\}\le\Pp\{|\xi_{11}|<e^{-\eta m}\}
\le2Ke^{-\eta m}.
\]
Both series are summable. Borel--Cantelli and a countable intersection over rational $\eta>0$
yield $R_m^{1/m}\to1$ almost surely. The two eventual inequalities follow by comparing the
limits $\omega(A_m(2))^{1/m}\to1/4$ and $R_m^{1/m}\to1$, as in the proof of
\corref{cor:phase-prob}.
\end{proof}
\end{revision}

\section{Further questions}

The argument replaces uniform control of exponentially many random hyperplanes by an averaged
geometric estimate and independent-probe amplification. This separation may be useful in other
extremal problems where a very large dependent family is tested by a smaller independent sample.

Two extensions are especially natural. First, one may ask whether subgaussian tails can be weakened
while retaining enough control of row norms and hyperplane normals. Second, for complex
\rev{i.i.d.} entries with a bounded two-dimensional density, the predicted exponent is
$-h(\gamma)/2$; proving the matching upper bound would require a complex analogue of the local and
pair estimates used here.

\begin{revision}
\section*{Acknowledgments and funding}
Y.~Shmalo acknowledges support from the European Research Council (ERC) under the European
Union's Horizon Europe programme (grant agreement No.~101041711), the Simons Foundation through
the Collaboration on the Theoretical Foundations of Deep Learning,
Heights Labs, Convex Nexus Capital, and the Israel Science Foundation (grant numbers 2258/19 and
4101/25).

\section*{Declaration of competing interest}
Y.~Shmalo holds equity interests in both Heights Labs and Convex Nexus Capital, which provided
research support, and discloses these relationships as potential financial competing interests.

\section*{Declaration of generative AI and AI-assisted technologies in the manuscript preparation process}
During the preparation of this work, the authors used generative AI tools to accelerate drafting and
revision. All mathematical claims, proofs, numerical interpretations, and bibliographic information
were subsequently reviewed and edited by the authors, who take full responsibility for the content
of the manuscript.
\end{revision}

\end{document}